\documentclass[9pt]{article}

\usepackage{amsmath}
\usepackage{amsfonts}
\usepackage{amsmath,amsthm}
\usepackage{amssymb}
\usepackage{latexsym}
\usepackage{amscd}
\usepackage{bm} 
\usepackage{stmaryrd} 
\usepackage[pdftex]{graphicx}

\theoremstyle{definition}
\newtheorem{thm}{\bf Theorem}[section]
\newtheorem{lem}[thm]{\bf Lemma}
\newtheorem{cor}[thm]{\bf Corollary}
\newtheorem{prop}[thm]{\bf Proposition}
\newtheorem*{acknowledge}{Acknowledgements}
\theoremstyle{definition}
\newtheorem{defn}[thm]{\bf Definition}
\theoremstyle{remark}

\newtheorem{claim}{\bf Claim}

\newtheorem{ques}[thm]{\bf Question}

\newcommand{\A}{\mathcal A}

\def\ri{\rightarrow}

\def\Om{\Omega}

\newcommand{\T}{\mathbb T}

\newcommand{\R}{\mathbb R}
\newcommand{\z}{\mathbb Z}

\def\im{\text{Im }}

\def\inner<#1>{\langle #1 \rangle}

\def\T{\mathcal T}
\DeclareMathOperator{\Spin}{Spin}

\begin{document}

\title{Rational homology 3-spheres and simply connected definite bounding}
\date{ }

\author{Kouki Sato and Masaki Taniguchi}

\maketitle
\begin{abstract}
For each rational homology 3-sphere $Y$ which bounds simply connected definite 4-manifolds of both signs,
we construct an infinite family of irreducible rational homology 3-spheres which are homology cobordant to $Y$ but cannot bound any simply connected definite 4-manifold.
As a corollary, for any 
coprime
integers $p,q$, we obtain an infinite family of irreducible rational homology 3-spheres 
which are homology cobordant to 
the lens space
$L(p,q)$ but cannot obtained by a knot surgery.
\end{abstract}


\section{Introduction}
Throughout this paper, all manifolds are assumed to be smooth, compact, orientable and oriented, and diffeomorphisms are orientation-preserving unless otherwise stated.

The intersection form of a $4n$ dimensional manifold has been used to study the topology of its boundary.
For instance, the first exotic 7-spheres discovered by Milnor \cite{Mi56} were distinguished by using the intersection form of 8-manifolds whose boundaries are the exotic 7-spheres.   
In the case of dimension 4, Donaldson's diagonalization theorem \cite{Do83} implies
that if a homology 3-sphere bounds a 4-manifold with non-diagonalizable
definite intersection form, then it cannot bound any rational homology 4-ball.

In light of the above results, for any 3-manifold $Y$, it seems natural to ask which 
bilinear forms
are realized by 
the intersection form of
a 4-manifold with boundary $Y$.
In the case where $Y$ is a rational homology 3-sphere, Choe and Park \cite{CP17} define $\T(Y)$
(resp.\ $\T^{TOP}(Y)$) as the set of all negative definite 
bilinear forms
realized by
the intersection form of 
a (resp.\ topological) 4-manifold with boundary $Y$, up to stable-equivalence. 
They prove in \cite{CP17} that $|\T^{TOP}(Y)|= \infty$ for any $Y$, while $|\T(Y)|< \infty$ if $\T(-Y)$ is not empty. Moreover, 
they show that either $\T(Y) \neq \emptyset$ or $\T(-Y) \neq \emptyset$ holds for any Seifert rational homology sphere $Y$.
Here we note that all 4-manifolds constructed in their proof of the above results are simply connected, and hence if we define $\T_s(Y)$ (resp.\ $\T_s^{TOP}(Y)$) by replacing ``4-manifolds''
in the definition of $\T(Y)$ (resp.\ $\T^{TOP}(Y)$) with ``simply connected 4-manifolds'',
then we can prove that 
$|\T_s^{TOP}(Y)|= \infty$ and $\T_s(-Y) \neq \emptyset \Rightarrow |\T_s(Y)|< \infty$ for any $Y$,
and either $\T_s(Y) \neq \emptyset$ or $\T_s(-Y) \neq \emptyset$ holds if $Y$ is Seifert fibered.
Then, how different are they? The aim of this paper is  
to prove the following theorem, which shows a big gap between $\T(Y)$ and $\T_s(Y)$.
\begin{thm}
\label{main thm}
For any rational homology 3-sphere $Y$ satisfying $\T_s(Y)\neq \emptyset $ and $\T_s(-Y)\neq \emptyset$, there exist infinitely many rational homology 3-spheres $\{Y_k\}_{k=1}^{\infty}$ 
which satisfy the following conditions.
\begin{enumerate}
\item $Y_k$ is homology cobordant to $Y$. 
\item $\T(Y_k)=\T(Y) \neq \emptyset$ and $\T(-Y_k)=\T(-Y) \neq \emptyset$.  
\item $\T_s(Y_k)=\emptyset $ and $\T_s(-Y_k)=\emptyset$.
\item If $k\neq k'$ then $Y_k$ is diffeomorphic to neither $Y_{k'}$ nor $-Y_{k'}$.
\item Each $Y_k$ is irreducible  and toroidal.
\end{enumerate}
\end{thm}
Here, rational homology 3-spheres $Y_0$ and $Y_1$ are {\it homology cobordant} if there exists a cobordism $W$ from $Y_0$ to $Y_1$ (i.e. $\partial W = (-Y_0) \amalg Y_1$) such that the inclusion $Y_{i} \hookrightarrow W$ induces an isomorphism between $H_*(Y_i;\z)$ and $H_*(W;\z)$ for each $i \in \{0,1\}$.
(Then we call $W$ a {\it homology cobordism}.)
We note that since $\T(Y)$ is invariant under homology cobordism (more generally, rational homology cobordism),
the first condition implies the second condition. Moreover, the third condition implies that any $Y_k$ is non-Seifert. 
We also note that there exist infinitely many rational homology 3-spheres satisfying
$\T_s(Y)\neq \emptyset $ and $\T_s(-Y)\neq \emptyset$.
For instance, any $p/q$ surgery of $S^3$ over any $0$-negative knot 
(defined in \cite{CHH13})
with $p/q>0$ satisfies this condition.
(In this case, there is a negative definite cobordism 
$W$
from the lens space $L(p,q)$ to such a $p/q$ surgery 
such that $i_*(\pi_1(L(p,q)))$ normally generates $\pi_1(W)$,
and $\T_s(L(p,q)) \neq \emptyset $.)

In order to prove Theorem \ref{main thm},
we first prove the following proposition, which is obtained by generalizing Auckly's construction in \cite{Au}.
\begin{prop}
\label{main prop}
For any rational homology 3-spheres $Y$ and $M$, 
there exist a rational homology 3-sphere $Y_M$ and a homology cobordism $W_M$
from $Y\#M\#(-M)$ to $Y_M$ 
which satisfy
\begin{enumerate} 
\item $i_*: \pi_1(Y_M) \to \pi_1(W_M)$ is surjective,
\item $i_*: \pi_1(Y\#M\#(-M)) \to \pi_1(W_M)$ is bijective, and
\item $Y_M$ is irreducible and toroidal,
\end{enumerate}
where $i_*$ denotes the induced homomorphism from the inclusion.
\end{prop}
Then, by assuming that $\T_s(Y)\neq \emptyset $, $\T_s(-Y)\neq \emptyset$ and $|\T_s(M) |>1$, and combining the first condition with Taubes's theorem in \cite{T87} (stated as Theorem \ref{taubes} in Section 3), we prove that $\T_s(Y_M)=\emptyset$ and $\T_s(-Y_M)= \emptyset$. 
Finally, we combine the second condition with the Chern-Simons invariants for 3-manifolds to find an infinite family $\{M_k\}_{k=1}^{\infty}$ of integer homology 3-spheres such that the 3-manifolds $\{Y_{M_k}\}_{k=1}^{\infty}$ are mutually distinct. (Note that if $M$ is an integer homology 3-sphere, then $Y\#M\#(-M)$ is homology cobordant to $Y$.)

As an application of Theorem \ref{main thm},
we provide a huge number of irreducible rational homology 3-spheres that are not obtained by  a knot surgery.
Here we note that if $Y$ is obtained by a knot surgery, 
then either $\T_s(Y)\neq\emptyset $ or $\T_s(-Y)\neq \emptyset$ holds (see \cite{OwSt12}).
Hence the 3-manifolds $\{Y_k\}_{k=1}^{\infty}$ in Theorem~\ref{main thm} are not obtained by  a knot surgery. 
Therefore, for instance, we have the following corollary.
\begin{cor}
\label{cor}
For any non-zero integers $p,q$, there are infinitely many irreducible rational homology 3-spheres which are homology cobordant to $L(p,q)$ but not obtained by a knot surgery.
\end{cor}
These are the first examples of irreducible rational homology 3-spheres which have non-trivial torsion first homology and are not obtained by a knot surgery.
Since infinitely many irreducible 3-manifolds with $H_1(Y) \cong \z$ which are not obtained by a knot surgery are given in \cite{HKMP18},  now we have infinitely many irreducible 3-manifolds
with $H_1(Y) \cong \z/p\z$ which are not obtained by a knot surgery for any integer $p$.

Finally we discuss some questions related to our results on knot surgery.
We first mention that it remains open whether the examples given in the proof of Corollary \ref{cor} have weight one fundamental group.
 \begin{ques}
 Do the examples given in the proof of Corollary \ref{cor} have
weight one fundamental group?
 \end{ques}
Next, while we have a huge number of irreducible rational homology 3-spheres which are not obtained by a knot surgery, all of our examples are toroidal.
Recently, Hom and Lidman \cite{HL18} provided infinitely many hyperbolic integral homology 3-spheres which are not obtained by surgery on a knot, while
the following question is still open.
 \begin{ques}
 Does there exist a hyperbolic rational homology $3$-sphere $Y$ 
such that $H_1(Y;\z)\neq 0$ and $Y$ is not obtained by a knot surgery?
 \end{ques}
In addition, our examples in Corollary \ref{cor} are homology cobordant to 
$L(p,q)$, and hence their $d$-invariants (defined in \cite{OzSz03}) satisfy 
$$
\{ d(Y,\frak{s}) \mid \frak{s} \in \Spin^c(Y) \} 
= \{ d(L(p,q), \frak{s}) \mid \frak{s} \in \Spin^c(L(p,q))\}
$$ for some $p, q$.
So we suggest the following question.

 \begin{ques}
 Does there exist a rational homology $3$-sphere $Y$ 
such that $
\{ d(Y,\frak{s}) \mid \frak{s} \in \Spin^c(Y) \} 
\neq \{ d(L(p,q), \frak{s}) \mid \frak{s} \in \Spin^c(L(p,q))\}
$ for any $p, q$ and $Y$ is not obtained by a knot surgery?
 \end{ques}

\begin{acknowledge}
The first author was supported by JSPS KAKENHI Grant Number 18J00808.
The second author was supported by JSPS KAKENHI Grant Number 17J04364.
\end{acknowledge}


\section{Proof of Proposition \ref{main prop}} 
In this section, we prove Proposition \ref{main prop}.
\setcounter{section}{1}
\setcounter{thm}{1}
\begin{prop}
For any rational homology 3-spheres $Y$ and $M$, 
there exist a rational homology 3-sphere $Y_M$ and a homology cobordism $W_M$
from $Y\#M\#(-M)$ to $Y_M$ 
which satisfy
\begin{enumerate} 
\item $i_*: \pi_1(Y_M) \to \pi_1(W_M)$ is surjective,
\item $i_*: \pi_1(Y\#M\#(-M)) \to \pi_1(W_M)$ is bijective, and
\item $Y_M$ is irreducible and toroidal,
\end{enumerate}
where $i_*$ denotes the induced homomorphism from the inclusion.
\end{prop}
\setcounter{section}{2}
\setcounter{thm}{0}

\def\proofname{Proof of Proposition \ref{main prop}}
\begin{proof}
\def\proofname{Proof}
We first describe the construction of $Y_M$ in the proposition.
Let $Y$ and $M$ be rational homology 3-spheres and $Y_M':= Y \# M \# (-M)$.
By the result of \cite{Tsu03}, there exists a null-homologous knot $K$ in $Y'_M$ whose complement $Y'_M\setminus K$ has a hyperbolic structure. We denote the exterior of $K$ in $Y'_M$ by $E_K$.
Let $C$ be a 3-manifold with torus boundary shown in Figure \ref{C}, where $\mu$ and
$\lambda$ in the figure are simple closed curves in the boundary of $C$.
 Define $Y_M$ as $C \cup_{\text{torus}} E_K$ by identification
\[
\begin{array}{lll}
\mu & \mapsto & \text{meridian} \\ 
\lambda & \mapsto & \text{preferred longitude}.
\end{array}
\]
It is easy to see that $Y_M$ is a rational homology 3-sphere.
\begin{figure}[htbp]
\begin{center}
\includegraphics[scale= 1.2]{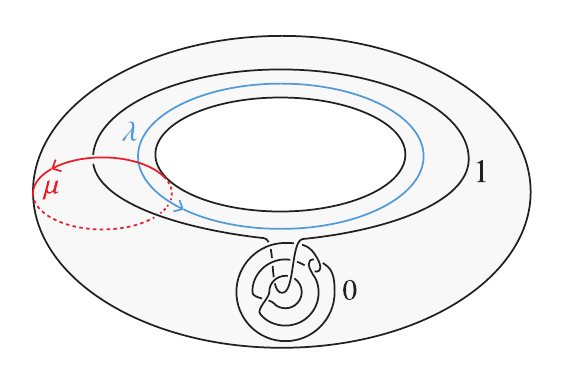}
\caption{\label{C} A 3-manifold $C$}
\end{center}
\end{figure}

\begin{claim}
\label{irr}
$Y_M$ is irreducible and toroidal.
\end{claim}

\begin{proof}
We make a similar argument to \cite{Au}.
More precisely, we use the following lemmas.

\begin{lem}[See p.13 in \cite{Au}]
\label{au lem}
Let $N$ be a 3-manifold and $F$ a properly embedded incompressible  surface in $N$.
If every component of $N-F$ is irreducible, then $N$ is also irreducible.
\end{lem}

\begin{lem}[Papakyriakopoulos' Loop Theorem]
\label{papa}
If the boundary of a 3-manifold $N$ is incompressible in $N$, then $i_*: \pi_1(\partial N) \to \pi_1(N)$ is injective.
\end{lem} 

\begin{lem}[See Theorem 2.6 in \cite{LS}, for example]
\label{ls lem}
Let
$$
\begin{CD}
C @>i_{A}>> A \\
@Vi_{B}VV @VVj_AV \\
B @>>j_{B}> A*_{C}B 
\end{CD}
$$
be the defining diagram of $A*_C B$. If $i_A$ and $i_B$ are injective, then
$j_A$ and $j_B$ are injective.
\end{lem}
By Lemma \ref{au lem}, it suffices to prove
\begin{enumerate}
\item $\partial C$ is incompressible in $Y_M$, and

\item both $C$ and $E_K$ are irreducible
\end{enumerate}
for proving Claim \ref{irr}. (Note that the first condition implies that $Y_M$ is a toroidal 3-manifold with essential torus $\partial C$.) Moreover, Lemma \ref{papa} and Lemma \ref{ls lem} implies that $\partial C$  is incompressible in $Y_M$ if $\partial C$ is incompressible both in $C$ and in $E_K$. To prove it, suppose that $\partial C$ is incompressible both in $C$ and in $E_K$.
Then it follows from Lemma \ref{papa} that both of the induced homomorphisms  $(i_C)_*: \pi_1(\partial C) \to \pi_1(C)$
and $(i_{E_K})_*: \pi_1(\partial C) \to \pi_1(E_K)$ are injective. In addition, Since $\pi_1(Y_M)=\pi_1(C)*_{\pi_1(\partial C)} \pi_1(E_K)$, Lemma \ref{ls lem} implies that
both of the induced homomorphisms $(j_C)_*: \pi_1(C) \to \pi_1(Y_M)$
and $(j_{E_K})_*: \pi_1(E_K) \to \pi_1(Y_M)$ are injective.
Now, assume that there exists a compressing disk for $\partial C$ in $Y_M$, and then $i_*: \pi_1(\partial C) \to \pi_1(Y_M)$ is not injective. However, since $i_*=(j_C)_* \circ (i_C)_*$ and  the right hand side is injective, it leads to a contradiction.

Here, we note that the 3-manifold $C$ is exactly the same as the manifold $C$ appearing  in \cite{Au},  and it is proved that $\partial C$ is  incompressible in $C$, and $C$ is irreducible.
Now let us prove that $\partial C = \partial E_K$ is incompressible in $E_K$, and $E_K$ is irreducible. The irreducibility of $E_K$ immediately follows from the fact that $E_K$ has a hyperbolic structure. 
Assume that there exists a compressing disk $D$ for $\partial E_K$ in $E_K$.
Then it follows from elementary arguments that $\partial D$ is a preferred longitude for $K$,
and hence $K$ bounds a disk in $Y'_M$. This implies that $E_K$ is homeomorphic to $Y'_{M} \# (S^1 \times D^2)$, and $E_K$ does not have any hyperbolic structure. This leads to a contradiction, and hence $\partial E_K$ is incompressible in $E_K$.
\end{proof}

Next, let $W_M$ denote a cobordism described by the relative Kirby diagram shown in Figure~\ref{W_M}. Here, the tangle diagram $\langle D \rangle$ in Figure \ref{W_M} is obtained as follows. 
We first take a diagram $D'$ of $K$ in $Y'_{M}$ (i.e. a knot diagram of $K$ in a surgery diagram of $Y'_M$) such that the linking number between $K$ and each component of the surgery link for $Y'_M$ is zero. 
Next, we derive a tangle diagram $D$ from $D'$ by removing a small disk whose intersection with $K$ is a small arc. Finally, by putting brackets around each surgery coefficient in $D$, we have the diagram $\langle D \rangle$.
Then, it follows from elementary handle theory that $W_M$ is a homology cobordism from $Y'_M$ to $Y_M$, and it admits a handle decomposition consisting of a single 1-handle $h^1$ and single 2-handle $h^2$.

\begin{figure}[htbp]
\begin{center}
\includegraphics[scale= 1.2]{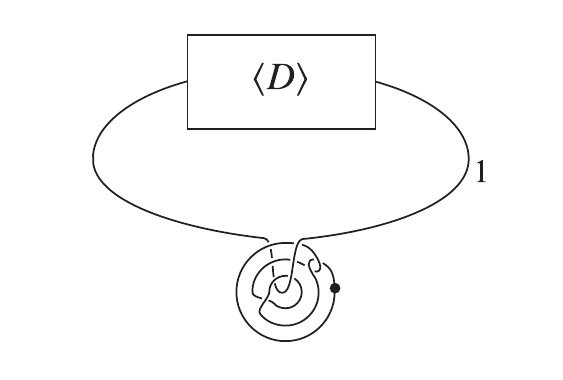}
\caption{\label{W_M} A cobordism $W_M$}
\end{center}
\end{figure}

\begin{claim}
\label{surj}
$i_*: \pi_1(Y_M) \to \pi_1(W_M)$ is surjective.
\end{claim}

\begin{proof}
By considering the dual decomposition, we have a handle decomposition of $W_M$ consisting a single 2-handle and single 3-handle. This implies that $i_*: \pi_1(Y_M) \to \pi_1(W_M)$ is surjective.
\end{proof}

\begin{claim}
\label{bij}
$i_*: \pi_1(Y'_{M}) \to \pi_1(W_M)$ is bijective.
\end{claim}

\begin{proof}
Let $\langle S \mid R \rangle$ be a presentation for $\pi_1(Y'_{M})=\pi_1(Y'_{M}\times [0,1])$
and $l$ a loop shown in Figure~\ref{loop}.
Then $\pi_1((Y'_{M}\times [0,1]) \cup h^1)$ is presented by $\langle S \cup \{ x \} \mid R \rangle$, where $x$ corresponds to $h^1$, and $\partial h^2$ is homotopic to $l$ in $(Y'_{M}\times [0,1]) \cup h^1$. This implies that the homotopy class of $\partial h^2$ is a word of the form $xw$, where $w$ is a word on $S$, and hence 
$\pi_1(W_M) = \pi_1((Y'_{M}\times [0,1]) \cup h^1 \cup h^2)$ is represented by
$\langle S \cup \{ x \} \mid R \cup \{ xw \} \rangle$. Moreover,  the diagram
$$
\begin{CD}
\langle S \mid R \rangle @> f >> \langle S \cup \{ x \} \mid R \cup \{ xw \} \rangle\\
@ V \cong VV @VV \cong V \\
\pi_1(Y'_{M}) @>>i_*> \pi_1(W_M)
\end{CD}
$$
is commutative, where $f$ maps $y\in S$ to $y$.
We construct the inverse of $f$. Define a map $g: \langle S \cup \{ x \} \mid R \cup \{ xw \} \rangle \to \langle S \mid R \rangle $ by
$$
y \mapsto \left\{
\begin{array}{llll}
y & (y \in S) \\ 
w^{-1}& (y=x) .
\end{array}
\right.
$$
Then it 
is
easy to see that $g$ is well-defined
and both $f \circ g$ and $g \circ f$ are the identity maps. (Note that $w^{-1}=x$ in $\langle S \cup \{ x \} \mid R \cup \{ xw \} \rangle.$)
This completes the proof.
\end{proof}

\begin{figure}[htbp]
\begin{center}
\includegraphics[scale= 1.2]{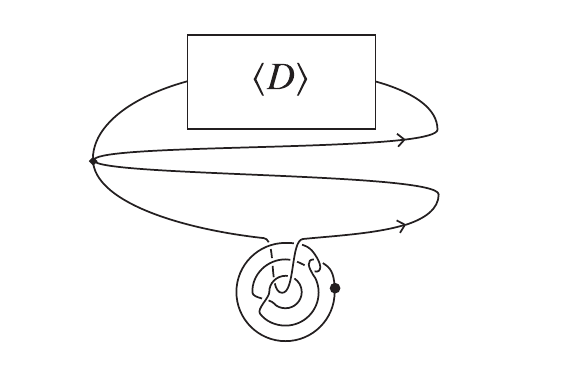}
\caption{\label{loop} A loop $l$ in 
$(Y'_M
\times [0,1]) \cup h^1$}
\end{center}
\end{figure}

The above arguments complete the proof of Proposition \ref{main prop}.
\end{proof}

\def\proofname{Proof}


\section{Definite bounding and Taubes's theorem}

In this section, 
we prove the following 
proposition 
by using a theorem of Taubes. 

\begin{prop}\label{definite}
Let $Y$ and $M$ be rational homology 3-spheres, and $Y_M$ a rational homology 3-sphere given by Proposition~\ref{main prop}.
If $\T_s(Y) \neq \emptyset$, $\T_s (-Y) \neq \emptyset$ and $|\T_s(M) |>1$, 
then we have $\T_s(Y_M)= \T_s (-Y_M) =\emptyset$.
\end{prop}
First we state Taubes's theorem. This is an end-periodic version of Donaldson's diagonalization theorem.

\begin{thm}\label{taubes} \cite{T87} Let $Y$ be a rational homology 3-sphere
 and $W:= K \cup_Y W_0 \cup_Y W_1 \cup_Y \cdots $
a connected non-compact 4-manifold
satisfying the following conditions. 
\begin{itemize} 
\item $K$ is a 
simply connected
negative definite 4-manifold with $\partial K =Y$.
\item $W_0$ is a negative definite cobordism from $Y$ to itself.
\item $W_i$ are copies of $W_0$ for any $i\in \{1,2, \ldots \}$.
\end{itemize} 
We also assume that there is no non-trivial representation from $\pi_1(W_0)$ to $SU(2)$. 
Then the intersection form of $K$ is diagonalizable.
\end{thm}
The assumption about $SU(2)$ representation is essential. 
If the assumption is removed, then one can easily find a counterexample to the theorem.
(For instance, take $Y=\Sigma(2,3,5)$ and $W_0= Y\times I$.)
This is an essential reason 
why we can claim the nonexistence of simply connected definite bounding. 

As a corollary of Theorem \ref{taubes}, we have the following lemma.
\begin{lem} \label{cpt} Let $M$ be a rational homology 3-sphere and $W$ is a negative definite cobordism from $M$ to itself. If $W$ is simply connected, then $|\T_s(M)|\leq 1$.
\end{lem}
\begin{proof} 
Suppose that $|\T_s(M)|> 1$. Then there exists a simply connected negative definite 4-manifold $K$ with 
$\partial K =M$ whose intersection form is not diagonalizable. However, the end-periodic manifold
$K \cup_M \cup W \cup_M W \cup_M \cdots $ satisfies 
all assumptions of Theorem \ref{taubes}.  This leads to a contradiction.
\end{proof}

Now we prove Proposition \ref{definite}.
\begin{proof}
We first assume that $\T_s(Y_M) \neq \emptyset$. 
Then there exists a simply connected negative definite 4-manifold $X$ with $\partial X= Y_M$.
Moreover, since $\T_s(-Y)\neq \emptyset$, $-Y$ also bounds a simply connected negative definite 4-manifold $U$. We use $X$, $U$ and $-W_M$ to construct a simply connected negative definite cobordism from $M$ to itself.

We first glue $X$ with $-W_M$ along $Y_M$, and denote it by $X'$. 
(Note that $\partial(-W_M)=-Y_M \amalg (Y\#M\#(-M))$.)
Then $X'$ is negative definite and $\partial X' = Y\#M\#(-M)$.
Furthermore, since $\pi_1(X)=1$ and $i_*: \pi_1(-Y_M) \to \pi_1(-W_M)$ is surjective, 
we have $\pi_1(X')=1$. Next, by attaching two 3-handles to $X'$, we obtain a 4-manifold $X''$
with $\partial X'' = Y \amalg M \amalg (-M)$. Finally, by gluing $X''$ 
with
$U$ along $Y$, 
we have a 4-manifold $W$ with boundary $M \amalg (-M)$. By the construction, it is easy to check that $\pi_1(W)=1$ and $W$ is negative definite.

Now, by applying Lemma \ref{cpt} to $W$, we conclude that $|\T_s(M)| \leq 1$.
However,  since $|\T_s(M)|>1$ is assumed, this leads to a contradiction.
As a consequence, we have $\T_s(Y_M) = \emptyset$.

If we assume $\T_s(-Y_M) \neq \emptyset$, then a similar argument gives $|\T_s(M)| \leq 1$,
which contradicts to the assumption $|\T_s(M)| > 1$. (In this case, use $\T_s(Y) \neq \emptyset$ and $W_M$ instead of $\T_s(-Y) \neq \emptyset$ and $-W_M$.)
This completes the proof.
\end{proof}


\section{Chern-Simons invariants} 
In this section, we give a method for finding an infinite family $\{ M_k \}$ such that  $\{Y_{M_k}\}$ are mutually disjoint. The goal of this section is to prove the following proposition. 
Here we denote the $(p,q,r)$-Brieskorn sphere by $\Sigma(p,q, r)$.
\begin{prop}\label{inf}
Let $Y$ be a rational homology 3-sphere, $p,q$ coprime integers, $M_n := \Sigma(p,q,pqn-1)$
and $Y_n:=Y_{M_n}$ a rational homology 3-sphere given by Proposition~\ref{main prop}.
Then there exists a numerical sequence $\{n_k\}_{k=1}^{\infty}$
such that 
if $k \neq k'$, then $Y_{n_k}$ is diffeomorphic to neither $Y_{n_{k'}}$ nor $-Y_{n_{k'}}$.
\end{prop}
In \cite{CG88}, it is shown that $|\T_s(\Sigma(p,q,pqn-1))|>1$. Hence we can apply Proposition~\ref{definite}
to $Y_{M_n}$ whenever $Y$ satisfies $\T_s(Y) \neq \emptyset$ and $\T_s (-Y) \neq \emptyset$.
In order to prove Proposition \ref{inf},
we use the Chern-Simons invariants for 3-manifolds.
Here we recall the Chern-Simons invariants.
For a given 3-manifold $Y$, let $P_Y$ be the product $SO(3)$ bundle.
First we introduce several definitions which are used for gauge theory.
We denote by $\text{Map}(Y,SO(3))$ the set of smooth maps from $Y$ to $SO(3)$.
The group structure on $SO(3)$ induces a group structure on $\text{Map}(Y,SO(3))$.
 Let $\A^f_Y$ be the set of $SO(3)$-flat connections on $P_Y$. Since $\text{Map}(Y,SO(3))$ can be identified with the set of automorphisms on $P_Y$, $\text{Map}(Y,SO(3))$ acts on $\A^f_Y$ by the pull-back of connections. The set of $SO(3)$-connections on $P_Y$ can be identified with the $\mathfrak{so}(3)$-valued $1$-forms on $Y$. Therefore we regard any element of $\A^f_Y$ as an element of $\Om^1(Y)\otimes  \mathfrak{so}(3)$.
Under these identifications, the action of $\text{Map}(Y,SO(3))$ on $\A^f_Y$ is written: 
\[
g^*a= g^{-1} d g+ g^{-1} ag, 
\]
where $a \in \Om^1(Y)\otimes  \mathfrak{so}(3)$.
This action defines the quotient
\[
{R}(Y):=\A^f_Y / \text{Map}(Y,SO(3)) .
\]
Then the Chern-Simons functional  
\begin{align}\label{cs}
\widetilde{cs}: \A^f_Y \ri \R
\end{align}
is defined by 
\[
\widetilde{cs}(a)= \frac{1}{8\pi^2} \int_Y \text{Tr}( a\wedge da + \frac{2}{3} a \wedge a\wedge a),
\]
where $a \in \Om^1(Y)\otimes  \mathfrak{so}(3)$. It is known that 
\[
\widetilde{cs}(g^*a) = \widetilde{cs}(a) + \text{deg} (g),
\]
where $g \in \text{Map}(Y,SO(3))$, $a \in \Om^1(Y) \otimes  \mathfrak{so}(3)$ and  $\text{deg} (g)$ is the mapping degree of $g$. Therefore the map \eqref{cs} descends the map: 
\[
 cs: {R}(Y) \ri \R/\z.
\]

Since the space $R(Y)$ is compact and the map $cs$ is locally constant, one can show the set 
$\im cs \subset \R/\z$ is a finite set.

By using the Chern-Simons functional, Furuta \cite{Fu90} defines a numerical invariant
$\epsilon$ as follows. (In \cite{FS90}, Fintushel and Stern also consider such an invariant.) Here we identify $(0,1]$ with $\R/\z$ via the quotient map $\R \ri \R/\z$ and regard $cs$ as a map from ${R}(Y)$ to $(0,1]$.

\begin{defn}For a 3-manifold $Y$, we define
\[
\epsilon(Y):= \left\{
\begin{array}{cl}\displaystyle
 \min_{a \in {cs}^{-1}(0,1)}  cs(a)  &  ({cs}^{-1}(0,1) \neq \emptyset) \\
 1 & ({cs}^{-1}(0,1) = \emptyset)
\end{array}\right.
\]

\end{defn}
There is a connected sum inequality for $\epsilon$ stated as follows. 
\begin{lem}
\label{sum}
For any two 3-manifolds $Y_1$ and $Y_2$, we have
$$
\epsilon(Y_1\# Y_2) \leq \min\{ \epsilon(Y_1), \epsilon(Y_2) \}.
$$
\end{lem}
\begin{proof}
For proving the lemma, it suffices to prove that 
$\epsilon(Y_1\# Y_2) \leq \epsilon(Y_1)$.
Let $\rho$ be an $SO(3)$ flat connection on $Y_1$ satisfying $cs(\rho)=\epsilon(Y_1)$
and $\theta$ the product connection on $Y_2$. 
By taking the connected sum of $\rho_M$ and $\theta$, we get an $SO(3)$ flat connection 
$\rho \# \theta$ over $Y_1 \# Y_2$.
Then it follows from the definitions of $cs$ and $\epsilon$ that 
$$
\epsilon(Y_1 \# Y_2) \leq cs(\rho \# \theta)=cs(\rho) = \epsilon(Y_1). 
$$

\end{proof}
Next, we prove the following lemma. This lemma says that if we have a nice cobordism, then we can estimate the value of $\epsilon$.
\begin{lem}\label{pi1}Let $Y_1$ and $Y_2$ be 3-manifolds.
Suppose that there is a cobordism $W$ from $Y_1$ to $Y_2$ such that 
$i_*: \pi_1(Y_1) \ri \pi_1(W)$ is bijective. Then the inequality 
\[
\epsilon(Y_2)\leq \epsilon(Y_1) 
\]
holds. 

\end{lem} 

\begin{proof} 
Suppose that $\rho$ is a $SO(3)$ flat connection satisfying $cs(\rho) =\epsilon(Y_1)$. 
Since $\pi_1(Y_1) \ri \pi_1(W)$ is bijective, we can extend $\rho$ over $W$ using the holonomy correspondence. We denote the extended connection by $\tilde{\rho}$. Then the equalities
\[
0=\frac{1}{8\pi^2} \int_{W} \text{Tr}( F(\rho) \wedge F(\rho))= cs(\rho)-cs(\tilde{\rho}|_{Y_2})
\]
hold. Therefore, we have
\[
 \epsilon(Y_2) \leq cs(\tilde{\rho}|_{Y_2})= cs(\rho)= \epsilon(Y_1).
 \]
\end{proof}
In our situation, we have the following estimate for $\epsilon( Y_M)$.  
\begin{cor}
\label{pi1 cor}
For any $Y$ and $M$, we have
$$
\epsilon(Y_M) \leq \epsilon(M).
$$ 
\end{cor}
\begin{proof}
By applying Lemma~\ref{pi1} to $W_M$, we have
$$
\epsilon(Y_{M}) \leq \epsilon(Y\#M \# (-M)) \leq \epsilon(M),
$$
where the second inequality follows from Lemma~\ref{sum}.
\end{proof}
Now, let us prove Proposition \ref{inf}.
\def\proofname{Proof of Proposition \ref{inf}}
\begin{proof}
It is proved by Furuta \cite{Fu90} and Fintushel-Stern \cite{FS90} that

\[
\epsilon(M_n) = \frac{1}{pq(pqn-1)}.
\]
Therefore, it follows from Corollary~\ref{pi1 cor} that for any $n$, we have
\[
\epsilon(Y_n) \leq \epsilon(M_n) = \frac{1}{pq(pqn-1)}.
\]
We construct a numerical sequence $\{n_k\}_{k=1}^{\infty}$ by induction.
First, we define $n_1:=1$. Next, suppose that $\{n_k\}_{k=1}^{m}$ is defined for some $m$. 
Since $\frac{1}{pq(pqn-1)} \to 0$ $(n \to \infty)$, there exists an integer $n$ such that
\[
\frac{1}{pq(pqn-1)} < \min_{1 \leq k \leq m}\{ \epsilon(Y_{n_k}), \epsilon(-Y_{n_k}) \}.
\]
Then we define $n_{m+1}:=n$. 

Now, let us prove that $\{n_k\}_{k}^{\infty}$ is the desired sequence. 
Suppose that $k \neq k'$.
Without loss of generality, we may assume that $k > k'$. Then, by the definition of $\{n_k\}_{k=1}^{\infty}$, the inequalities
\[
\epsilon(Y_{n_{k}}) \leq \frac{1}{pq(pqn_{k}-1)} < \min_{l < k}\{ \epsilon(Y_{n_l}), \epsilon(-Y_{n_l}) \}
\]
hold. In particular, since $k'<k$, we have
\[
\epsilon(Y_{n_{k}}) < \min \{ \epsilon(Y_{n_{k'}}), \epsilon(-Y_{n_{k'}}) \}.
\]
This proves that $Y_{n_k}$ is diffeomorphic to neither $Y_{n_{k'}}$ nor $-Y_{n_{k'}}$.
\end{proof}
\def\proofname{Proof}

\section{Proof of Main Theorem}

In this section, we prove Theorem~\ref{main thm},
which is stated as follows.
\setcounter{section}{1}
\begin{thm}
For any rational homology 3-sphere $Y$ satisfying $\T_s(Y)\neq \emptyset $ and $\T_s(-Y)\neq \emptyset$, there exist infinitely many rational homology 3-spheres $\{Y_k\}_{k=1}^{\infty}$ 
which satisfy the following conditions.
\begin{enumerate}
\item $Y_k$ is homology cobordant to $Y$. 
\item $\T(Y_k)=\T(Y) \neq \emptyset$ and $\T(-Y_k)=\T(-Y) \neq \emptyset$.  
\item $\T_s(Y_k)=\emptyset $ and $\T_s(-Y_k)=\emptyset$.
\item If $k\neq k'$ then $Y_k$ is diffeomorphic to neither $Y_{k'}$ nor $-Y_{k'}$.
\item Each $Y_k$ is irreducible  and toroidal.
\end{enumerate}
\end{thm}
\setcounter{section}{5}
\begin{proof}
Let $M_n:=\Sigma(2,3,6n-1)$, $Y_{M_n}$ be a rational homology 3-sphere given by Propositon~\ref{main prop} and $\{n_k\}_{k=1}^{\infty}$ a numerical sequence given by 
Proposition~\ref{inf}.
Then we define $Y_k := Y_{M_{n_{k}}}$. Let us prove that $\{Y_k\}_{k=1}^{\infty}$is the desired family.

First, since $Y_k$ is homology cobordant to $Y \# M_{n_k} \# (-M_{n_k})$ and $M_{n_k}$ is an integer homology 3-sphere, $Y_k$ is homology cobordant to $Y$ for any $k$.

Second, since $\T$ is a homology cobordism invariant, $\T(Y) \supset \T_s(Y)$, and we assume that $\T_s(Y) \neq \emptyset$ and $\T_s(-Y) \neq \emptyset$, both
$\T(Y_k)=\T(Y) \neq \emptyset$ and $\T(-Y_k)=\T(-Y) \neq \emptyset$ hold for any $k$.

Third, since $|M_{n_k}|>1$ holds for any $k$, it follows from 
Proposition~\ref{definite} that $\T_s(Y_k)=\emptyset $ and $\T_s(-Y_k)=\emptyset$.

Fourth, it follows from Proposition~\ref{inf}
that if $k\neq k'$ then $Y_k$ is diffeomorphic to neither $Y_{k'}$ nor $-Y_{k'}$.

Finally, it follows from Proposition \ref{main prop} that
each $Y_k$ is irreducible and toroidal.
This completes the proof.
\end{proof}

\bibliographystyle{plain}
\bibliography{tex}

\end{document}